\documentclass[reqno,11pt]{amsart}
\usepackage{amsmath, latexsym, amsfonts, amssymb, amsthm, amscd}
\usepackage{multirow}
\usepackage{hyperref}
\usepackage{mathrsfs}
\usepackage{comment}
\usepackage{xcolor}

 \usepackage{graphicx}
\usepackage[utf8]{inputenc} 

\usepackage{datetime}

\usepackage{cancel}
\usepackage{graphics,epsf,psfrag}
\numberwithin{equation}{section}

\newtheorem{theorem}{Theorem}[section]
\newtheorem{lemma}[theorem]{Lemma}
\newtheorem{example}[theorem]{Example}

\newtheorem{cor}[theorem]{Corollary}
\newtheorem{rem}[theorem]{Remark}
\newtheorem{definition}[theorem]{Definition}

\definecolor{forestgreen}{rgb}{0.13, 0.55, 0.13}

\newcommand{\R}{\mathbb{R}}

\DeclareMathSymbol{\leqslant}{\mathalpha}{AMSa}{"36} 
\DeclareMathSymbol{\geqslant}{\mathalpha}{AMSa}{"3E} 
\DeclareMathSymbol{\eset}{\mathalpha}{AMSb}{"3F}     
\renewcommand{\leq}{\;\leqslant\;}                   
\renewcommand{\geq}{\;\geqslant\;}                   

\newcommand{\norm}[1]{\left\lVert#1\right\rVert}

\title[Weak solutions for singular multiplicative SDEs]{Weak solutions for singular multiplicative SDEs via regularization by noise}

\author{Florian Bechtold, Martina Hofmanov\'a}
\address{Fakultät für Mathematik, Universität Bielefeld, 33501 Bielefeld, Germany}
\email{fbechtold@math.uni-bielefeld.de, hofmanova@math.uni-bielefeld.de}
\begin{document}
\begin{abstract}
    We study multiplicative SDEs perturbed by an additive fractional Brownian motion on another probability space. Provided the Hurst parameter is chosen in a specified regime, we establish existence of probabilistically weak solutions to the SDE if the measurable diffusion coefficient merely satisfies an integrability condition. In particular, this allows to consider certain singular diffusion coefficients. 
\end{abstract}
\maketitle

\section{Introduction}
Consider the classical problem 
\[
dX_t=b(X_t)dt+\sigma(X_t)dB_t, \qquad X_0=x_0\in \R^d,
\]
where $B$ is an $n$-dimensional Brownian motion. As it is well-known, provided the coefficients $b:\R^d\to \R^d$ and $\sigma:\R^d\to \mathbb{M}^{d\times n}$ are Lipschitz continuous the above stochastic differential equation admits a unique strong solution \cite{Revuz1999}, \cite{oksendal2003}. If $b$, $\sigma$ are only continuous and satisfy a growth condition \cite{skoro}, \cite{Hofmanov2012}  or a Lyapunov condition \cite{Hofmanov20133}, it is still possible to obtain weak solutions. This can be done in passing by a tightness argument: Mollifying $b$ and $\sigma$, we obtain in a first step a sequence $(X^\epsilon)_\epsilon$ of associated solutions for which we can show tightness in appropriate spaces. By martingale arguments, it is then possible to identify all the terms with the corresponding limit in a second step, concluding the consideration. 
\newline
\newline
In the present work, we establish that the conditions imposed in particular on the diffusion coefficient $\sigma$ that ensure the existence of weak solutions can be further relaxed in the presence of an additional additive fractional Brownian motion. Our approach mainly relies on a robustification of the above two step strategy harnessing the regularizing effects due to local times of fractional Brownian motion and averaging operators \cite{harang2020cinfinity}, \cite{gubicat}, \cite{galeati2020noiseless}. For the reader's convenience, let us briefly recall the definition of the latter and how they are employed to establish regularization by noise results for ordinary differential equations. \newline\newline
For a $d$-dimensional fractional Brownian motion $w^H$ and some smooth nonlinearity $b:\R^d\to \R^d$, consider the problem

\[
Y_t=y_0+\int_0^t b(Y_s)ds-w^H_t.
\]
By a simple transformation, we see that $Y$ solves the above if and only if $X=Y+w^H$ solves
\begin{equation}
    X_t=y_0+\int_0^t b(X_s-w^H_s)ds.
    \label{regular ode}
\end{equation}
Heuristically, as one expects the oscillations of the fractional Brownian motion to dominate the oscillations of $X$, we may conjecture that locally, i.e. for $0<t-s\ll1$, the above Lebesgue integral should behave as
\begin{equation}
    \label{local approx}
    \int_s^t b( X_r-w^H_r)dr\simeq \int_s^tb( X_s-w^H_r)dr=:(T^{-w^H}_{s,t}b)(X_s),
\end{equation}
where $T^{-w^H}b$ denotes the averaging operator, which averages out the function $b$ by integrating along the highly fluctuating paths of $-w^H$. As it was remarked in \cite{gubicat}, \cite{galeati2020noiseless}, this averaging gives rise to a quantifiable regularization effect in the case of fractional Brownian motion. Alternatively to studying the regularization effect on the averaging operator directly, Harang and Perkowski \cite{harang2020cinfinity} observed that if $w^H$ admits a local time $L$, the above averaging operator may be rewritten thanks to the occupation times formula\footnote{We refer to Appendix for the basic definitions for local times and the occupation times formula .} as
\[
(T^{-w^H}_{s,t}b)(X_s)=\int_s^tb( X_s-w^H_r)dr=(b*L_{s,t})(X_s).
\]
By Young's inequality in Besov spaces \cite{kuhn2021convolution}, the spatial regularity of the averaging operator is therefore the sum of the regularity of $b$ and $L$, yielding also a quantifiable regularization. Let us remark at this point that while the results in \cite{galeati2020noiseless} are sharper, allowing even to consider non-autonomous equations, the approach in \cite{harang2020cinfinity} has the merit of being pathwise stable in the sense Theorem \ref{regularity of averaging operator} (see also  Remark \ref{comparaison averaging local time}). In our analysis, we will make use of both approaches and their respective merits.\\
\\
In both approaches this gain of regularity on the level of averaging operators can subsequently be exploited to reconstruct the Lebesgue integral in \eqref{regular ode} by ``sewing together" the local approximations in \eqref{local approx} by means of the Sewing Lemma\footnote{For the readers convenience, we state the Sewing Lemma in the corresponding section of the Appendix.}. Eventually, this allows for the study of an abstract non-linear Young problem
\begin{equation}
   X_t=y_0+(\mathcal{I}A^X)_t,
   \label{nonlinear-young}
\end{equation}
where $\mathcal{I}$ denotes the Sewing operator and $A^X$ the germ
\[
A_{s,t}^X=(T^{-w^H}_{s,t}b)(X_s).
\]
Remark that due to the quantifiable regularization of the averaging operator, \eqref{nonlinear-young} might be well posed even in the case of continuous or merely distributional $b$, therefore establishing a regularization by noise phenomenon with respect to the ordinary differential equation. 
\newline\newline
In the following, we intend to study
\[
Y_t=x_0+\int_0^t \sigma(Y_s)dB_s-w^H_t,
\]
or rather, again via the above transformation,
\begin{equation}
    X_t=x_0+\int_0^t\sigma(X_s-w^H_s)dB_s.
    \label{main}
\end{equation}
Since pathwise regularization by noise for singular drifts as sketched above is by now quite well understood, we set $b=0$ for readability focusing on the diffusion term.  However upon combining it with results from \cite{galeati2020noiseless}, \cite{harang2020cinfinity}, the same approach applies mutatis mutandis to \eqref{main} with a general, distributional drift $b$.\\
\\
Regularization by noise for multiplicative SDEs where the above Brownian motion $B$ is replaced by another fractional Brownian motion $\beta^{H'}$ of Hurst parameter $H'>1/2$ was recently studied by Galeati and Harang \cite{galeati2020regularizationn}. The main idea in \cite{galeati2020regularizationn} is that due to a probabilistic result in \cite{hairer2019averaging}, one can introduce an averaged multiplicative field
\[
(\Gamma^{-w^H}_t\sigma)(x):=\int_0^t \sigma(x-w^H_s)d\beta^{H'}_s
\]
whose regularity properties are inherited from the regularity properties of the averaged field $T^{-w^H}\sigma$ introduced above. 
While very powerful allowing in particular to consider distributional $\sigma$, the above strategy does not apply in the present setting of classical Brownian motion where $H'=1/2$. Refer to Remark \ref{comparison} for a further comparison of our results with the ones obtained in \cite{galeati2020regularizationn}.
\newline\newline 
In contrast with \cite{galeati2020regularizationn}, our considerations are mainly based on classical Ito-calculus, which we further robustify thanks to the regularizing effect of the averaging operator $(T^{-w^H}_t \mbox{Tr}(\sigma^* \sigma))=\norm{\sigma}_{HS}^2*L_{t}$: Essentially, in going through the Ito-isometry, we exploit a regularization effect on the second moment of the diffusion term.  A natural broad function space to consider that is also well behaved under the operator of ``taking the square" is therefore $\sigma\in L^p_x$. Notice however that as  $\mbox{Tr}(\sigma \sigma^*)$ has to be well-defined for this approach, we are unable to study regularization effects on distributional scales as in \cite{galeati2020regularizationn}. We obtain the following main result.
\begin{theorem}
\label{main theorem}
Let $w^H$ be a $d$-dimensional fractional Brownian motion of Hurst parameter $H$ on $(\Omega^H, \mathcal{F}^H, \mathbb{P}^H)$. For $p\in [2,\infty)$, let $\sigma\in L^p_x$. Suppose that $d/p<1$. Let $H$ be such that
\[
2H<(1+\frac{d}{(p/2)\wedge 4/3})^{-1}.
\]
Then there exists a $\mathcal{F}^H$ measurable set $\Gamma(\sigma)\subset \Omega^H$ such that $\mathbb{P}^H(\Gamma(\sigma))=1$ and such that for any $\omega^H\in \Gamma(\sigma)$ the stochastic differential equation
\[
X_t=x_0+\int_0^t\sigma(X_s-w^H_s(\omega^H))dB_s
\]
admits a weak solution on any bounded interval $[0,T]$.
\end{theorem}

\begin{rem}
Our result relies on the regularizing of the square of the Hilbert-Schmidt norm $\|\sigma\|_{HS}^{2}$. Thus, the same proof applies also to an infinite dimensional cylindrical Wiener process $B$ under an appropriate Hilbert-Schmidt assumption on $\sigma$, while keeping the dimension $d$ of the solution and the fractional Brownian motion finite.
\end{rem}

 Let us illustrate the above result with an example.

\begin{example}
Let $n=d$ and consider for $\gamma>0$ and $K>0$
\[
\sigma(x)=\mbox{Id}\frac{1}{|x|^\gamma}1_{|x|\leq K}.
\]
Note that then $\sigma\in L^p_{x}$ provided $\gamma<d/p$. The condition $d/p<1$ appearing above means that we need to impose $\gamma<1$. In order to ensure that we may choose $p\geq 2$, we need to demand further $2<d/\gamma$, which represents a constraint only in the case $d=1$. Overall, we obtain that for
\[
2H<(1+\frac{d}{(d/(2\gamma))\wedge 4/3})^{-1}
\]
the SDE
\[
X_t=x_0+\int_0^t\left(\frac{1}{|X_s-w^H_s|^\gamma}1_{|X_s-w^H_s|\leq K}\right)dB_s
\]
admits a weak solution.
\end{example}
\begin{rem}[Comparison with results in \cite{galeati2020regularizationn}]
Let $H'\in (1/2, 1)$ and let $\beta^{H'}$ be a fractional Brownian motion on $(\Omega', \mathcal{F}', \mathbb{P}', (\mathcal{F}'_t)_t)$.  Note that for $\sigma\in L^p_x$, combining \cite[Theorem 1]{galeati2020regularizationn} with Theorem \ref{regularity of averaging operator I} below, one obtains that provided 
\[
H<\frac{H'-1/2}{2+d/p},
\]
the problem
\[
X_t=x_0+\int_0^t \sigma(X_s-w^H_s)d\beta^{H'}_s
\]
admits a unique solution in the sense that for any $x_0\in \R^d$ any two pathwise solutions solutions defined on $(\Omega', \mathcal{F}', \mathbb{P}')$ are indistinguishable. There holds moreover path-by-path uniqueness (refer to \cite[Definition 40]{galeati2020regularizationn}).
\label{comparison}
\end{rem}

\subsection*{Notation}
We endow the space of $d\times n$ matrices $\mathbb{M}^{d\times n}$ with the Hilbert-Schmidt norm
\[
\norm{A}_{HS}:=\sqrt{\mbox{Tr}(A^* A)}.
\]
For a function $\sigma:\R^d\to \mathbb{M}^{d\times n}$, we denote by $L^p_x$ the Bochner space equipped with the norm
\[
\norm{\sigma}_{L^p_x}=\left( \int_{\R^d}\norm{\sigma(x)}_{HS}^pdx\right)^{1/p}.
\]

\section{Tightness}
For the readers convenience we begin by citing two results on the regularity of averaging operators associated with fractional Brownian motion $T^{-w^H}f$. As we assume throughout $H<1/d$, $\mathbb{P}^H$-almost any realization of the fractional Brownian motion admits a local time $L$ allowing to write 
\[
T^{-w^H}f: (t, x)\to \int_0^t f(x-w^H_s)ds=(f*L_t)(x).
\]
The results cited represent a slight adaptation to our purposes. In particular, Theorem~\ref{regularity of averaging operator I} provides a more quantitative statement on the Hölder-continuity of the averaging operator than provided in the reference, which is however immediate from the proof of \cite[Theorem 4]{galeati2020noiseless}. 


\begin{theorem}[Regularity of averaging operators I, { \cite[Theorem 3.4]{harang2020cinfinity}}]
Let $w^H$ be $d$-dimensional fractional Brownian motion of Hurst parameter $H$ on $(\Omega^H, \mathcal{F}^H, \mathbb{P}^H)$ such that $H<1/d$ and let $p \in [1, \infty)$. Then there exists a $\mathcal{F}^H$ measurable set $\Gamma\subset \Omega^H$ such that $\mathbb{P}^H(\Gamma)=1$ and such that for any $\omega^H\in\Gamma$, $w^H(\omega^H)$ admits a local time $L$ and for any \[
\lambda<1/(2H)-d/(p\wedge 2), \qquad \gamma<1-(\lambda+d/2)H\] we have $T^{-w^H}f\in C^\gamma_t C^{\lambda}_x$  provided $\lambda\in \R^+\backslash \mathbb{N}$ and $f\in L^p_x$. Moreover we have the stability property
\begin{equation}
    \norm{T^{-w^H}(f_1-f_2)}_{C^\gamma_t C^{\lambda}_x}\lesssim \norm{f_1-f_2}_{L^p_x} .
    \label{stability}
\end{equation}
\label{regularity of averaging operator}
\end{theorem}

\begin{theorem}[Regularity of averaging operators II, { \cite[Theorem 4]{galeati2020noiseless}}]
Let $w^H$ be $d$-dimensional fractional Brownian motion of Hurst parameter $H$ on $(\Omega^H, \mathcal{F}^H, \mathbb{P}^H)$. For $p \in [1, \infty)$ let $f\in L^{p}_x$. Then there exists a $\mathcal{F}^H$ measurable set $\Gamma(f)\subset \Omega^H$ such that $\mathbb{P}^H(\Gamma(f))=1$ and such that for any $\omega^H\in\Gamma(f)$ and
 \[
\lambda<1/(2H)-d/p, \qquad\gamma<1-(\lambda+d/p)H\]
we have $T^{-w^H}f\in C^\gamma_t C^{\lambda}_x$  provided $\lambda\in \R^+\backslash \mathbb{N}$. Moreover we have the stability property
\begin{equation}
    \mathbb{E}^H\left[\norm{T^{-w^H}(f_1-f_2)}_{C^\gamma_t C^{\lambda}_x}^{2k}\right]\lesssim \norm{f_1-f_2}_{L^p_x}^{2k} .
    \label{stabilityII}
\end{equation}
\label{regularity of averaging operator I}
\end{theorem}

\begin{rem}
Since we use a similar argumentation throughout the paper to get various bounds on the parameters, let us outline how the conditions on parameters in Theorem~\ref{stabilityII} are obtained from \cite[Theorem 4]{galeati2020noiseless}.
We apply  \cite[Theorem 4]{galeati2020noiseless} with $s=0$, $q=\infty$ and $\rho = \lambda + m$ and use the Sobolev embeddding for Bessel potential spaces
$L_{x}^{\lambda+ m, p} \subset C_{x}^{\lambda}$ which requires $m = d / p+\varepsilon$ for some $\varepsilon>0$.
Then the condition in \cite[Theorem 4]{galeati2020noiseless} reads
\[ H (\lambda + m) < 1 / 2 \]
hence
\begin{equation}\label{eq:1}
 \lambda < 1 / (2 H) - d/p-\varepsilon.
 \end{equation}
Therefore, whenever $\lambda$ satisfies the condition in Theorem~\ref{stabilityII}, there is an $\varepsilon$ so that the condition \eqref{eq:1} is satisfied for some $\varepsilon$.
\end{rem}

\begin{rem}
Note that Theorem \ref{regularity of averaging operator} is weaker than Theorem \ref{regularity of averaging operator I} in terms of the regularization effect. However, it allows for the pathwise stability property \eqref{stability}, compared to the stability result \eqref{stabilityII} that employs moments and thus implicitly necessitates some measurability in $\omega^H$. In particular, the zero-set outside of which Theorem \ref{regularity of averaging operator} holds is indepent of the considered function $f$, whereas in Theorem \ref{regularity of averaging operator I}, it depends on $f$. As we will lose measurability in $\omega^H$ by going through a tightness argument, we resort to Theorem \ref{regularity of averaging operator} for the most part of our analysis. 

\label{comparaison averaging local time}
\end{rem}
Throughout the remainder of the paper, let us  fix some realization $w^H(\omega^H)$ of a fractional Brownian motion on $(\Omega^H, \mathcal{F}^H, \mathbb{P}^H)$ that permits to employ the regularity statements of Theorems \ref{regularity of averaging operator} and \ref{regularity of averaging operator I} for an accordingly prescribed Hurst parameter. More precisely, by Theorems \ref{regularity of averaging operator} and \ref{regularity of averaging operator I}, we find a $\Gamma(\sigma)\subset \Omega^H$ such that $\mathbb{P}^H(\Gamma(\sigma))=1$ and such that for any $\omega^H\in \Gamma(\sigma)$ the averaging operator enjoys the cited regularity.

\begin{lemma}[A priori bounds for solutions to mollified problem]
\label{a priori bound}
Let $p\in [2, \infty)$ and let  $\sigma:\R^d\to \mathbb{M}^{d\times n}$ satisfy  $\sigma\in L^p_x$. Let $\sigma_\epsilon$ be a  mollification of some smooth cut-off of $\sigma$, i.e. for a sequence of mollifiers $(\rho^\epsilon)_\epsilon$ and for $\varphi^\epsilon\in C^\infty_0$ with $\mbox{supp}(\varphi^\epsilon)\subset B_0(1/\epsilon)$, set $\sigma_\epsilon=\rho^\epsilon*(\varphi^\epsilon\cdot \sigma)$. Let $B^\epsilon$ be an $(\Omega^\epsilon, \mathcal{F}^\epsilon, \mathbb{P}^\epsilon, (\mathcal{F}^\epsilon_t)_t)$ Brownian motion. Suppose that $2H<(1+\frac{d}{(p/2)\wedge( 4/3)})^{-1}$. Let $X^\epsilon$ be the unique strong solution to the problem
\[
X^\epsilon_t=x_0+\int_0^t \sigma_\epsilon(X^\epsilon_r-w^H_r(\omega^H))dB^\epsilon_r.
\]
Then for any $m\geq 4$ and $\gamma_0<1-Hd/2$ we have 
\[
\sup_{s\neq t\in [0,T]}\frac{\mathbb{E}|X^\epsilon_{s,t}|^m}{|t-s|^{m\gamma_0/2}}<\infty
\]
and
\[
\mathbb{E}[\sup_{t\in [0,T]} |X^\epsilon_{t}|^m]<\infty
\]
uniformly in $\epsilon>0$. 
\end{lemma}
\begin{proof}
Let $m\geq 4$. Note that by the Burkholder-Davis-Gundy inequality, we have that for $\epsilon>0$ fixed
\begin{equation}
\begin{split}
\label{BDG}
\mathbb{E}[|X_{s,t}^\epsilon|^m]&\lesssim \mathbb{E}\left[\left(\int_s^t \norm{\sigma_\epsilon( X^\epsilon_r-w^H_r)}_{HS}^2dr\right)^{m/2}\right]\\
&\lesssim \max_{1\leq i,j\leq d}\norm{\sigma_\epsilon^{ij}}_{L^\infty_{x}}^m |t-s|^{m/2},
\end{split}
\end{equation}
i.e. the constant
\[
c_{m,\epsilon, \gamma_0}^m:=\sup_{s\neq t\in [0,T]}\frac{\mathbb{E}[|X_{s,t}^\epsilon|^m]}{|t-s|^{m\gamma_0/2}}<\infty
\]
is finite for any fixed $\epsilon>0$ and $\gamma_0\leq 1$ but a priori diverging in $\epsilon$. For  $\gamma_0<1-Hd/2$ we now use the stochastic Sewing Lemma \cite{stochasticsewing}(see also Lemma~\ref{stoch sewing a priori}) to give an alternative definition to the above expression that allows to show that $c_{m,\epsilon, \gamma_0}$ is uniformly bounded in $\epsilon>0$. Define the germ
\[
A_{s,t}^\epsilon:=(\norm{\sigma_\epsilon}_{HS}^2*L_{s,t})(X^\epsilon_s)=(T^{-w^H}_{s,t}\norm{\sigma_\epsilon}_{HS}^2)(X^\epsilon_s),
\]
where by $\norm{\sigma_\epsilon}_{HS}^2$ we understand the mapping $x\to \norm{\sigma_\epsilon(x)}_{HS}^2$, which by assumption is in $L^{p/2}_x$ uniformly in $\epsilon$.  We note that by Theorem \ref{regularity of averaging operator} and the condition imposed on the Hurst parameter $H$, we have for some small $\delta>0$ and $\gamma_0, \gamma_1\in (1/2, 1)$ such that $\gamma_0/2+\gamma_1>1$ and 
 \begin{equation*}
     \norm{(T^{-w^H}_{s,t}\norm{\sigma_\epsilon}_{HS}^2)}_{C^\delta_x}\lesssim |t-s|^{\gamma_0}, \qquad \norm{(T^{-w^H}_{s,t}\norm{\sigma_\epsilon}_{HS}^2)}_{C^{1+\delta}_x}\lesssim |t-s|^{\gamma_1}
 \end{equation*}
 uniformly in $\epsilon>0$. Indeed, remark that the condition $2H< (1+\frac{d}{(p/2)\wedge 2})^{-1}$ ensures $A^\epsilon_t\in C^{1+\delta}$ whereas $2H< (1+\frac{d}{4/3})^{-1}$ ensures that $\gamma_0, \gamma_1$ may be chosen such that $\gamma_0/2+\gamma_1>1$. The maximal $\gamma_0$ we may choose for such $H$ is precisely the one we give in the condition. Using the notation of the sewing Lemma introduced in the Appendix, we thus have 
\begin{align*}
    |A^\epsilon_{s,t}|\leq \norm{(T_{s,t}^{-w^H}\norm{\sigma_\epsilon}_{HS}^2)}_{L^\infty_x}
    \lesssim|t-s|^{\gamma_0},
\end{align*}
as well as 
\begin{align*}
\begin{aligned}
    |(\delta A^\epsilon)_{s,u, t}|&=|(T_{u,t}^{-w^H}\norm{\sigma_\epsilon}_{HS}^2)(X^\epsilon_s)-(T_{u,t}^{-w^H}\norm{\sigma_\epsilon}_{HS}^2)(X^\epsilon_u)|
   \\& \lesssim \norm{(T_{u,t}^{-w^H}\norm{\sigma_\epsilon}_{HS}^2)}_{C^1_x}|X_u^\epsilon-X_s^\epsilon|.
    \end{aligned}
\end{align*}
 We have by Jensens' inequality that 
\begin{align*}
    \norm{\mathbb{E}[(\delta A^\epsilon)_{s,u,t} | \mathcal{F}^\epsilon_s]}_{L^{m/2}(\Omega^\epsilon)}&\lesssim (\mathbb{E}[|(\delta A^\epsilon)_{s,u,t}|^{m/2}])^{2/m}\\
   &\lesssim (\mathbb{E}[|(\delta A^\epsilon)_{s,u,t}|^{m}])^{1/m}\\
   &\lesssim (\mathbb{E}[|((t-u)^{\gamma_1}|X^\epsilon_u-X^\epsilon_s|)|^{m}])^{1/m}\\
    &\lesssim |t-u|^{\gamma_{1}} \mathbb{E}[|X_u^\epsilon-X_s^\epsilon|^{m}]^{1/m}\\
    &\lesssim c_{m,\epsilon, \gamma_0}  |t-u|^{\gamma_{1}} |u-s|^{\gamma_0/2}
\end{align*}
As by our assumption $\gamma_{1}+\gamma_0/2>1$, the above shows that $A^\epsilon$ admits a stochastic sewing $\mathcal{A}^\epsilon$. The stochastic sewing Lemma \cite{stochasticsewing} furthermore implies that 
\[
\norm{\mathcal{A}^\epsilon_{s,t}}_{L^{m/2}(\Omega^\epsilon)}\lesssim \norm{A^\epsilon_{s,t}}_{L^{m/2}(\Omega^\epsilon)}+c_{m,\epsilon, \gamma_0}|t-s|^{\gamma_{1}+\gamma_0/2} \lesssim |t-s|^{\gamma_0}(1+c_{m,\epsilon, \gamma_0} T^{\gamma_{1}-\gamma_0/2}).
\]
Moreover, due to the regularity of $\sigma_\epsilon$, we may identify the sewing with the Lebesgue integral in \eqref{BDG} in the sense that
\begin{align*}
    \norm{\mathcal{A}^\epsilon_{s,t}}_{L^{m/2}(\Omega^\epsilon)}^{m/2}=\mathbb{E}\left[\left(\int_s^t \norm{\sigma_\epsilon( X^\epsilon_r-w^H_r)}_{HS}^2dr\right)^{m/2}\right].
\end{align*}
We refer to Lemma \ref{identification} the Appendix for a proof of this equality. Overall, going back to \eqref{BDG}, this yields
\begin{align*}
    \mathbb{E}[|X^\epsilon_{s,t}|^m]\lesssim  \norm{\mathcal{A}^\epsilon_{s,t}}_{L^{m/2}(\Omega^\epsilon)}^{m/2}\lesssim |t-s|^{m\gamma_0/2}(1+c_{m,\epsilon, \gamma_0}^{m/2})
\end{align*}
and therefore
\[
c_{m,\epsilon, \gamma_0}^m\lesssim 1+c_{m,\epsilon, \gamma_0}^{m/2}
\]
yielding that indeed $c_{m,\epsilon, \gamma_0}$ is uniformly bounded in $\epsilon$. 
\end{proof}

Note that by the Kolmogorov continuity theorem, we obtain the following.

\begin{cor}[Tightness]
\label{tightness}
We have for any $\epsilon>0$ that $X^\epsilon\in C^{\gamma_0/2}([0,T], \R^d)$ and moreover, for any $m\geq 1$ we have
\begin{equation}
\label{a priori for vitali}
\mathbb{E}\left[\sup_{t\neq s\in[0,T]}\left|\frac{X^\epsilon_{s,t}}{|t-s|^{\gamma_0/2}}\right|^m\right]<\infty
\end{equation}
uniformly in $\epsilon>0$. Hence the laws of and $(X^\epsilon, B^\epsilon)_\epsilon$ are tight\footnote{Remark that strictly speaking, we loose an (arbitrarily small) amount of regularity in applying the Kolmogorow continuity theorem and the compact embedding $C^{\gamma_0/2}\hookrightarrow^c C^{\gamma_0/2-\delta}$ for $\delta>0$. However, as $\gamma_0$ needs to satisfy a strict upper bound, we can incorporate these small losses of regularity into the strict inequality satisfied by $\gamma_0$.} on $C^{\gamma_0/2}([0,T], \R^d)\times C^{\gamma_0/2}([0,T], \R^n)$.
\end{cor}
By Prokhorov's and  Skorokhod's theorem, we may therefore conclude that :
\begin{cor}[Extraction of a convergent subsequence]
There exists a probabilistic basis $(\bar{\Omega}, \bar{\mathcal{F}}, \bar{\mathbb{P}})$, processes $(\bar{X^\epsilon}, \bar{B^\epsilon})$ on the said basis whose laws coincide with those of $(X^\epsilon, B^\epsilon)$, and a process $(\bar{X}, \bar{B})$  such that 
\[
(\bar{X^\epsilon}, \bar{B^\epsilon})\to (\bar{X}, \bar{B})
\]
$\bar{\mathbb{P}}$-almost surely in $C([0,T], \R^d)\times C([0,T], \R^n)$. Moreover, $\bar{B}$ and $\bar{B}^\epsilon$ are $(\bar{\Omega}, \bar{\mathcal{F}}, \Bar{\mathbb{P}})$-Brownian motions. Denote by $(\bar{\mathcal{F}}_t)_t$ the augmentation of the filtration generated by $(\bar{X}, \bar{B})$. Finally, by weak-*-lower semicontinuity of norms and Fatou's Lemma we have for any $m\geq 1$
\begin{align}
   \bar{\mathbb{E}}[|\bar{X}_{s,t}|^m]\lesssim \liminf_{\epsilon\to 0}   \bar{\mathbb{E}}[|\bar{X}^\epsilon_{s,t}|^m]\lesssim |t-s|^{m\gamma_0/2}
   \label{continuity of limit process}
\end{align}
from which we conclude by Kolmogorov's continuity theorem again that $\bar{X}\in C^{\gamma_0/2}$, $\bar{\mathbb{P}}$-almost surely and that 
\begin{equation}
\label{integralbility hoelder constant}
\bar{\mathbb{E}}\left[\sup_{t\neq s\in[0,T]}\left|\frac{\bar{X}_{s,t}}{|t-s|^{\gamma_0/2}}\right|^m\right]<\infty.
\end{equation}
\end{cor}

\section{Identification of the limit}
Assuming the conditions in Lemma \ref{a priori bound} hold, we  proceed by a stochastic compactness argument \cite{BreitFeireislHofmanova}, adapted to the present context. Throughout this section, let $(e_j)_{j=1, \dots, d}$ be the canonical basis of $\R^d$ and $(f_i)_{i=1, \ldots, n}$ be the canonical basis of $\R^n$. We know that the processes 
\begin{equation}
    \begin{split}
         t &\to M^{j,\epsilon}_{t}:= \langle  \bar{X}^\epsilon_t, e_{j}\rangle-\langle x_0, e_j\rangle=\langle e_j,  \int_0^t \sigma_\epsilon( \bar{X}^\epsilon_r-w^H_r)d\Bar{B}_r^\epsilon\rangle , \\
    t &\to (M^{j,\epsilon}_{t})^{2}-\int_0^t | \sigma_\epsilon^*( \bar{X}^\epsilon_r-w^H_r) e_j |^2dr,\\
   t &\to M^{j,\epsilon}_t  \langle \bar{B}^\epsilon_t, f_i\rangle -\int_0^t\langle f_i,  \sigma_\epsilon^*( \bar{X}^\epsilon_r-w^H_r)e_j \rangle dr
    \end{split}
    \label{start}
\end{equation}
are martingales with respect to $(\bar{\mathcal{F}}_t)_t$. Note that since $\sigma_\epsilon$ is smooth, we have by Lemma~\ref{identification} that
\begin{align*}
    &\int_0^t | \sigma_\epsilon^*( \bar{X}^\epsilon_r-w^H_r) e_j |^2dr=(\mathcal{I}A^{j,\epsilon})_t,\\
    &\int_0^t\langle f_i,  \sigma_\epsilon^*( \bar{X}^\epsilon_r-w^H_r)e_j \rangle dr=(\mathcal{I}a^{i,j,\epsilon})_t,
\end{align*}
where 
\begin{align*}
    A_{s,t}^{j,\epsilon}&=(|\sigma_\epsilon^* e_j|^2*L_{s,t})(\bar{X}_s^\epsilon)\\
    a^{i,j,\epsilon}_{s,t}&=(\sigma^{ji}_\epsilon*L_{s,t})(\bar{X}_s^\epsilon).
\end{align*}
Note that \eqref{start} being martingales is equivalent to having that for any bounded continuous functional $\phi$ on $C([0,s], \R^d)\times C([0,s], \R^n)$: 
\begin{equation}
    \begin{split}
        \bar{\mathbb{E}}[\phi(\bar{X}^\epsilon|_{[0,s]}, \bar{B}^\epsilon|_{[0,s]}) (M^{j, \epsilon}_t-M^{j, \epsilon}_s)]&=0,\\
        \bar{\mathbb{E}}[\phi(\bar{X}^\epsilon|_{[0,s]}, \bar{B}^\epsilon|_{[0,s]}) ((M^{j, \epsilon}_{t})^2-(M^{j, \epsilon}_{s})^2-(\mathcal{I}A^{j,\epsilon})_{s,t})]&=0,\\
        \bar{\mathbb{E}}[\phi(\bar{X}^\epsilon|_{[0,s]}, \bar{B}^\epsilon|_{[0,s]}) (M^{j,\epsilon}_{t}\langle\bar{B}_{t}^{\epsilon},f_{i}\rangle-M^{j,\epsilon}_{s}\langle\bar{B}_{s}^{\epsilon},f_{i}\rangle-(\mathcal{I}a^{i,j,\epsilon})_{s,t})]&=0.
        \label{martingale_epsilon}
    \end{split}
\end{equation}
We next intend to pass to the limit in \eqref{martingale_epsilon}. Note that by almost sure convergence and \eqref{a priori for vitali} all the terms with the exception of the appearing sewings converge due to  Vitali's convergence theorem. For the sewings $\mathcal{I}A^{j,\epsilon}$ and $\mathcal{I}a^{i,j,\epsilon}$ we shall employ Lemma \ref{sewing convergence} adapted to the stochastic sewing setting. 
\begin{lemma}
\label{step one conv}
For $s< t\in [0,T]$ and $m\geq 1$, we 
\begin{align*}
    \norm{(\mathcal{I}A^{j,\epsilon})_{s,t}-(\mathcal{I}A^j)_{s,t}}_{L^m(\bar{\Omega})}&\to 0\\
    \norm{(\mathcal{I}a^{i,j,\epsilon})_{s,t}- (\mathcal{I}a^{i,j})_{s,t}}_{L^m(\bar{\Omega})}&\to 0,
\end{align*}
where 
\begin{align*}
    A_{s,t}^{j}&=(|\sigma^* e_j|^2*L_{s,t})(\bar{X}_s)\\
    a^{i,j}_{s,t}&=(\sigma^{ji}*L_{s,t})(\bar{X}_s).
\end{align*}
\end{lemma}
\begin{proof}
The assertion follows from the stability of averaging operators cited in \eqref{regularity of averaging operator} and the fact that $\sigma_\epsilon\to \sigma $ in $L^p_x$ and $\bar{X}^\epsilon\to \bar{X}$ in $C([0,T], \R^d)$ as well as Lemma \ref{sewing convergence}.\\
\\
Indeed, observe that due to \eqref{a priori for vitali} we have that
\begin{align*}
    \norm{(\delta a^{i,j, \epsilon})_{s,u,t}}_{L^m(\bar{\Omega})}&=\norm{(\sigma^{ji}_\epsilon*L_{u,t})(\bar{X}^\epsilon_u)-(\sigma^{ji}_\epsilon*L_{u,t})(\bar{X}^\epsilon_s)}_{L^m(\bar{\Omega})}\\
    &\lesssim\bar{\mathbb{E}}\left[\sup_{t\neq s\in[0,T]}\left|\frac{\bar{X}^\epsilon_{s,t}}{|t-s|^{\gamma_0/2}}\right|^m\right]^{1/m}\norm{\sigma_\epsilon}_{L^p_x}|t-s|^{\gamma_0/2+\gamma_1}\\
    &\lesssim \norm{\sigma}_{L^p_x}|t-s|^{\gamma_0/2+\gamma_1}
\end{align*}
uniformly in $\epsilon>0$. Moreover, also by \eqref{a priori for vitali} and Vitali's theorem, we have that actually $X^\epsilon\to X$ in $L^m(\bar{\Omega}, C([0,T], \R^d))$. We therefore observe that
\begin{align*}
    \norm{a^{i,j}_{s,t}-a^{i,j, \epsilon}_{s,t}}_{L^m(\bar{\Omega})}&\leq \norm{(\sigma^{ji}*L_{s,t})(\bar{X}_s)-(\sigma^{ji}*L_{s,t})(\bar{X}^\epsilon_s)}_{L^m(\bar{\Omega})}\\
    &+\norm{(\sigma^{ji}*L_{s,t})(\bar{X}_s^\epsilon)-(\sigma^{ji}_\epsilon*L_{s,t})(\bar{X}^\epsilon_s)}_{L^m(\bar{\Omega})}\\
    &\lesssim |t-s|^{\gamma_1}\bar{\mathbb{E}}[\norm{\bar{X}^\epsilon-\bar{X}}_{\infty}^m]^{1/m}+\norm{\sigma-\sigma^\epsilon}_{L^p}|t-s|^{\gamma_0/2}.
\end{align*}
By Lemma \ref{sewing convergence}, this implies that indeed $\mathcal{I}a^{i,j, \epsilon}\to \mathcal{I}a^{i,j}$ in $C^{\gamma_0/2 \wedge \gamma_1}_tL^m(\bar{\Omega})$ and thus in particular the claim. Similar considerations hold for $A^{j, \epsilon}$ by remarking that $$\norm{|\sigma^*_\epsilon e_j|^2-|\sigma^*e_j|^2}_{L^{p/2}_x}\to 0.$$
\end{proof}
By the above Lemma, we may now pass to the limit in \eqref{martingale_epsilon}, obtaining:
\begin{cor}
\label{step two conv}
For $i=1, \dots, n$ and $j=1, \dots, d$, the processes 
\begin{equation}
    \begin{split}
          t &\to M^{j}_{t}:=\langle e_j, \bar{X}_t-x_0\rangle, \\
  t &\to (M^j_t)^2-(\mathcal{I}A^j)_t,\\
t &\to M^{j}_t  \langle \bar{B}_t, f_i\rangle -(\mathcal{I}a^{i,j})_t,
    \end{split}
    \label{martingale-limit}
\end{equation}
are martingales with respect to the filtration $(\bar{\mathcal{F}}_t)_t$. 
\end{cor}
In order to conclude that indeed $M$ coincides with the corresponding stochastic integral in the limit, we need to extend \cite[Proposition A.1]{HOFMANOVA20134294} to our sewing setting. That is precisely the content of the next Lemma \ref{martingale identification}, a technical part of which we moved into the subsequent Lemma \ref{identificationII} for the sake of readability.

\begin{lemma}
Suppose that for $i=1, \dots, n$ and $j=1, \dots, d$, the processes in \eqref{martingale-limit} are martingales. Then we have 
\[
M_t=\int_0^t \sigma(\bar{X}_s-w^H_s)d\bar{B}_s.
\]
\label{martingale identification}
\end{lemma}

\begin{proof}
We show that for any $j=1, \ldots, d$ 
\[
\bar{\mathbb{E}}[\langle M_t-\int_0^t \sigma(\bar{X}_r-w^H_r)d\bar{B}_r, e_j\rangle  ^2]=0.
\]
Let $\sigma_\epsilon$ be again mollifications multiplied with smooth cut off functions. Note that by definition (refer to Lemma \ref{extending the integral}), 
\[
\lim_{\epsilon\to 0}\bar{\mathbb{E}}[\langle\int_0^t \sigma(\bar{X}_s-w^H_s)d\bar{B}_s-\int_0^t \sigma_\epsilon(\bar{X}_s-w^H_s)d\bar{B}_s, e_j\rangle ^2]=0.
\]
Hence, it suffices to show
\[
\bar{\mathbb{E}}[\langle M_t-\int_0^t \sigma_\epsilon(\bar{X}_s-w^H_s)d\bar{B}_s, e_j\rangle^2]\to 0.
\]
Note that by our previous findings, we have
\begin{align*}
&\bar{\mathbb{E}}[\langle M_t-\int_0^t \sigma_\epsilon(\bar{X}_s-w^H_s)d\bar{B}_s, e_j\rangle ^2]\\
=&\bar{\mathbb{E}}[\langle M_t, e_j\rangle^2]+\bar{\mathbb{E}}[\langle \int_0^t \sigma_\epsilon(\bar{X}_s-w^H_s)d\bar{B}_s, e_j\rangle^2]
-2\bar{\mathbb{E}}[\langle M_t, e_j\rangle \langle \int_0^t \sigma_\epsilon(\bar{X}_s-w^H_s)d\bar{B}_s, e_j\rangle ]\\
=&\bar{\mathbb{E}}[(\mathcal{I}A^{j})_t]+\bar{\mathbb{E}}[\int_0^t |\sigma_\epsilon^*(\bar{X}_s-w^H_s)e_j|^2ds]-2\bar{\mathbb{E}}[\langle M_t, e_j\rangle \langle \int_0^t \sigma_\epsilon(\bar{X}_s-w^H_s)d\bar{B}_s, e_j\rangle ],
\end{align*}
where we recall that 
\[
A^j_{s,t}=(|\sigma^* e_j|^2*L_{s,t})(\bar{X}_s)=((\sigma \sigma^*)^{jj}*L_{s,t})(\bar{X}_s).
\]
Note moreover that we have, again by identification of the Lebesgue integral with the Sewing in the smooth setting
\[
\bar{\mathbb{E}}[\int_0^t |\sigma_\epsilon^*(\bar{X}_s-w^H_s)e_j|^2ds]=\bar{\mathbb{E}}[(\mathcal{I}F^{j,\epsilon})_{s,t}],
\]
where
\[
F^{j,\epsilon}_{s,t}=((\sigma_\epsilon \sigma_\epsilon^*)^{jj}*L_{s,t})(\bar{X}_s).
\]
By Lemma \ref{identificationII} we have
\[
\bar{\mathbb{E}}[\langle M_t, e_j\rangle \langle \int_0^t \sigma_\epsilon(\bar{X}_s-w^H_s)d\bar{B}_s, e_j\rangle ]=\bar{\mathbb{E}}[(\mathcal{I}G^{j, \epsilon})_t],
\]
where
\[
G^{j,\epsilon}_{s,t}=((\sigma \sigma_\epsilon^*)^{jj}*L_{s,t})(\bar{X}_s).
\]
We therefore conclude that 
\begin{align*}
    &\bar{\mathbb{E}}[\langle M_t-\int_0^t \sigma_\epsilon(\bar{X}_s-w^H_s)d\bar{B}_s, e_j\rangle ^2]\\
    =&\bar{\mathbb{E}}[(\mathcal{I}A^{j})_t]+\bar{\mathbb{E}}[(\mathcal{I}F^{j,\epsilon})_{s,t}]-2\bar{\mathbb{E}}[(\mathcal{I}G^{j, \epsilon})_t]
\end{align*}
and since the Sewings are stable under $\epsilon \to 0$ thanks to the stability property in Lemma~\ref{regularity of averaging operator} and Lemma \ref{sewing convergence}, we may indeed conclude our claim that 
\[
M_t=\int_0^t \sigma(\bar{X}_s-w^H_s)d\bar{B}_s.
\]

\end{proof}
In summary, this concludes the proof of Theorem \ref{main theorem}. 
\begin{lemma}
\label{identificationII}
Suppose the conditions of Lemma \ref{a priori bound}. Let $(e_j)_{j=1, \dots, d}$ be a canonical basis of $\R^d$. For $\epsilon>0$ fixed and $d/p<1$ we have for any $j=1, \dots d$
\[
\bar{\mathbb{E}}[\langle M_t, e_j\rangle \langle  \int_0^t \sigma_\epsilon(\bar{X}_s-w^H_s)d\bar{B}_s, e_j\rangle ]=\bar{\mathbb{E}}[(\mathcal{I}G^{j, \epsilon})_t],
\]
where
\[
G^{j,\epsilon}_{s,t}=((\sigma \sigma_\epsilon^*)^{jj}*L_{s,t})(\bar{X}_s).
\]
\end{lemma}
\begin{proof}
Note that $\mathbb{P}^H$ almost surely the process 
\[
t\to \sigma_\epsilon(\bar{X}_t-w^H_t)
\]
is progressively measurable and in $L^2(\Omega\times [0,T])$. Hence we can approximate it by elementary processes, i.e. take 
\[
\sigma_{\epsilon, N}(s):=\sum_{i=1}^N \sigma_\epsilon(\bar{X}_{t_i}-w^H_{t_i})1_{[t_i, t_{i+1})}(s)
\]
where $s=t_1<t_2<\dots <t_{N+1}=t$. Let $(f_k)_{k=1, \ldots, n}$ be the canonical basis of $\R^n$. Then 
\begin{align*}
    &\bar{\mathbb{E}}[\langle M_t-M_s, e_j\rangle \langle  \int_s^t \sigma_{\epsilon, N}(r)d\bar{B}_r, e_j\rangle  |\bar{\mathcal{F}}_s]\\
    =& \bar{\mathbb{E}}[ \sum_i^N \langle M_{t_{i+1}}-M_{t_i}, e_j\rangle \langle (\bar{B}_{t_{i+1}}-\bar{B}_{t_i}), \sigma_\epsilon^*(\bar{X}_{t_i}-w^H_{t_i}) e_j\rangle  |\bar{\mathcal{F}}_s]\\
    =& \sum_i^N \sum_{k=1}^n\bar{\mathbb{E}}[\langle f_k, \sigma^*_\epsilon(\bar{X}_{t_i}-w^H_{t_i}) e_j\rangle \bar{\mathbb{E}}[\langle M_{t_{i+1}}, e_j\rangle \langle \bar{B}_{t_{i+1}}, f_k\rangle-\langle M_{t_{i}}, e_j\rangle \langle \bar{B}_{t_{i}}, f_k\rangle|\bar{\mathcal{F}}_{t_i}]|\bar{\mathcal{F}}_s] \\
    =& \sum_i^N\sum_{k=1}^n\bar{\mathbb{E}}[\langle f_k, \sigma^*_\epsilon(\bar{X}_{t_i}-w^H_{t_i}) e_j\rangle\bar{\mathbb{E}}[(\mathcal{I}a^{k,j})_{t_i, t_{i+1}}|\bar{\mathcal{F}}_{t_i}]\bar{\mathcal{F}}_{s}]\\
    =& \bar{\mathbb{E}}[\sum_i^N\sum_{k=1}^n\langle f_k, \sigma^*_\epsilon(\bar{X}_{t_i}-w^H_{t_i}) e_j\rangle(\mathcal{I}a^{k,j})_{t_i, t_{i+1}})|\bar{\mathcal{F}}_s].
\end{align*}
Upon taking expectation we obtain
\begin{align*}
    &\bar{\mathbb{E}}[\langle M_t-M_s, e_j\rangle \langle  \int_s^t \sigma_{\epsilon, N}(r)d\bar{B}_r, e_j\rangle]\\
    &=\sum_i^N\bar{\mathbb{E}}[\sum_{k=1}^n\langle f_k, \sigma^*_\epsilon(\bar{X}_{t_i}-w^H_{t_i}) e_j\rangle(\sigma^{jk}*L_{t_i, t_{i+1}})(\bar{X}_{t_i})]\\
    &+\sum_i^N\bar{\mathbb{E}}[\sum_{k=1}^n\langle f_k, \sigma^*_\epsilon(\bar{X}_{t_i}-w^H_{t_i}) e_j\rangle\left((\mathcal{I}a^{k,j})_{t_i, t_{i+1}})-(\sigma^{jk}*L_{t_i, t_{i+1}})(\bar{X}_{t_i})\right)].
\end{align*}
Remark that by the (classical) Sewing Lemma in combination with \eqref{integralbility hoelder constant}, the second sum above vanishes in the limit $N\to \infty$. Notice moreover that 
\[
\lim_{N\to \infty}\sum_i^N\bar{\mathbb{E}}[\sum_{k=1}^n\langle f_k, \sigma^*_\epsilon(\bar{X}_{t_i}-w^H_{t_i}) e_j\rangle(\sigma^{jk}*L_{t_i, t_{i+1}})(\bar{X}_{t_i})],
\]
if convergent, is by definition nothing but the sewing of the germ 
\[
Z_{s,t}^{j, \epsilon}:=\bar{\mathbb{E}}[\sum_{k=1}^n\langle f_k, \sigma^*_\epsilon(\bar{X}_{s}-w^H_{s}) e_j\rangle(\sigma^{jk}*L_{s,t})(\bar{X}_{s})].
\]
Indeed, $Z^{j, \epsilon}$ admits a sewing as
\begin{align*}
    &\delta Z_{s,u,t}^{j, \epsilon}\\
    &=\bar{\mathbb{E}}[\sum_{k=1}^n\langle f_k, \sigma^*_\epsilon(\bar{X}_{s}-w^H_{s}) e_j\rangle\left((\sigma^{jk}*L_{u,t})(\bar{X}_{s})-(\sigma^{jk}*L_{u,t})(\bar{X}_{u})\right)]\\
    &+\bar{\mathbb{E}}[\sum_{k=1}^n\langle f_k, \sigma^*_\epsilon(\bar{X}_{s}-w^H_{s})-\sigma^*_\epsilon(\bar{X}_{u}-w^H_{u}) e_j\rangle(\sigma^{jk}*L_{u,t})(\bar{X}_{s})\\
    &\lesssim |t-s|^{\gamma_0/2+\gamma_1}+\sum_{k=1}^n|t-s|^{H\wedge \gamma_0/2}\norm{\sigma^{jk}*L_{u,t}}_{L^\infty}\\
    &=|t-s|^{\gamma_0/2+\gamma_1}+\sum_{k=1}^n|t-s|^{H\wedge \gamma_0/2}\norm{T^{-w^H}_{u,t}\sigma^{jk}}_{L^\infty}.
\end{align*}
where in the above inequality, we exploited \eqref{continuity of limit process} and again  $T^{-w^H}\sigma^{jk}\in C^{\gamma_1}_tC^{1}_x$, as well as the Lipschity-continuity of $\sigma_\epsilon$ (this is where the mollification is required). Moreover, a competition between the term $\sigma_\epsilon(\bar{X}_s-w^H_s)$ - which is more regular in time, provided $H$ large  - and $\norm{T^{-w^H}_{u,t}\sigma^{jk}}_{L^\infty}$ - which is more regular, provided $H$ small - occurred. By Theorem \ref{regularity of averaging operator I}, we know that 
 \[
 \norm{T^{-w^H}_{u,t}\sigma^{jk}}_{L^\infty}\lesssim |t-u|^{1-\frac{d}{p}H-\eta}
 \]
 for any $\eta>0$, as $\sigma^{jk}$ is time independent. Note that it is at this point that we need to employ Theorem \ref{regularity of averaging operator I} in order to not loose too much time regularity of the averaging operator that would be otherwise problematic in the 'regularity competition' above. This is also the reason the set $\Gamma$ in our main result \ref{main theorem} depends on $\sigma$. As by assumption $d/p<1$, we have that 
 \[
 H+1-d/pH-\eta>1
 \]
 for $\eta>0$ sufficiently small. Moreover, by choosing the maximal $\gamma_0$ available in Lemma~\ref{a priori bound}, namely $\gamma_0=1-Hd/2-\eta$ for any $\eta>0$, we have 
 \[
\frac{1}{2}(1-Hd/2-\eta)+1-Hd/p>1
 \]
which is satisfied under the condition
 \[
2H<(d/4+d/p)^{-1}.
 \]
 Remark however that for $d/p<1$, this condition on the Hurst parameter is already satisfied under the assumptions of Lemma \ref{a priori bound}. Overall, we conclude that $Z^{j, \epsilon}_{s,t}$ admits a sewing. Hence, we have
\[
\lim_{N\to \infty}\sum_i^N\bar{\mathbb{E}}[\sum_{k=1}^n\langle f_k, \sigma^*_\epsilon(\bar{X}_{t_i}-w^H_{t_i}) e_j\rangle(\sigma^{jk}*L_{t_i, t_{i+1}})(\bar{X}_{t_i})]=(\mathcal{I}Z^{j, \epsilon})_{t}. 
\]
Finally, again by a mollification argument on $\sigma$  similar to Lemma \ref{identification} one verifies that 
\[
(\mathcal{I}Z^{j,\epsilon})_{t}=(\mathcal{I}G^{j, \epsilon})_t,
\]
concluding the statement.
\end{proof}

\section{Appendix}
\subsection*{Local time and occupation times formula}
We recall for the reader the basic concepts of occupation measures, local times and the occupation times formula. A comprehensive review paper on these topics is \cite{horowitz}. 
\begin{definition}
Let $w:[0,T]\to \R^d$ be a measurable path. Then the occupation measure at time $t\in [0,T]$, written $\mu^w_t$ is the Borel measure on $\R^d$ defined by 
\[
\mu^w_t(A):=\lambda(\{ s\in [0,t]:\ w_s\in A\}), \quad A\in \mathcal{B}(\R^d),
\]
where $\lambda$ denotes the standard Lebesgue measure. 
\end{definition}
The occupation measure thus measures how much time the process $w$ spends in certain Borel sets. Provided for any $t\in [0,T]$, the measure is absolutely continuous with respect to the Lebesgue measure on $\R^d$, we call the corresponding Radon-Nikodym derivative local time of the process $w$:
\begin{definition}
Let $w:[0,T]\to \R^d$ be a measurable path. Assume that there exists a measurable function $L^w:[0,T]\times \R^d\to \R_+$ such that 
\[
\mu^w_t(A)=\int_A L^w_t(z)dz, 
\]
for any $A\in \mathcal{B}(\R^d)$ and  $t\in [0,T]$. Then we call $L^w$ local time of $w$. 
\end{definition}
Note that by the definition of the occupation  measure, we have for any bounded measurable function $f:\R^d\to \R$ that 
\begin{equation}
    \int_0^tf(w_s)ds=\int_{\R^d} f(z)\mu^w_t(dz).
    \label{occupation times formula}
\end{equation}
The above equation \eqref{occupation times formula} is called occupation times formula. Remark that in particular, provided $w$ admits a local time, we also have for any $x\in \R^d$
\begin{equation}
    \int_0^tf(x-w_s)ds=\int_{\R^d} f(x-z)\mu^w_t(dz)=\int_{\R^d}f(x-z)L^w_t(z)dz=(f*L^w_t)(z).
\end{equation}
\subsection*{The Sewing Lemma and its stochastic version}
We recall the Sewing Lemma due to \cite{gubi} (see also \cite[Lemma 4.2]{frizhairer}) as well as its stochastic version due to \cite{stochasticsewing}. Let $E$ be a Banach space, $[0,T]$ a given interval. Let $\Delta_n$ denote the $n$-th simplex of $[0,T]$, i.e. $\Delta_n:\{(t_1, \dots, t_n)| 0\leq t_1\dots\leq t_n\leq T \} $. For a function $A:\Delta_2\to E$ define the mapping $\delta A: \Delta_3\to E$ via
\[
(\delta A)_{s,u,t}:=A_{s,t}-A_{s,u}-A_{u,t}
\]
Provided $A_{t,t}=0$ we say that for $\alpha, \beta>0$ we have $A\in C^{\alpha, \beta}_2(E)$ if $\norm{A}_{\alpha, \beta}<\infty$ where
\[
\norm{A}_\alpha:=\sup_{(s,t)\in \Delta_2}\frac{\norm{A_{s,t}}_E}{|t-s|^\alpha}, \qquad \norm{\delta A}_{\beta}:=\sup_{(s,u,t)\in \Delta_3}\frac{\norm{(\delta A)_{s,u,t}}_E}{|t-s|^\beta} \qquad \norm{A}_{\alpha, \beta}:=\norm{A}_\alpha+\norm{\delta A}_\beta
\]
For a function $f:[0,T]\to E$, we note $f_{s,t}:=f_t-f_2$

Moreover, if for any sequence $(\mathcal{P}^n([s,t]))_n$ of partitions of $[s,t]$ whose mesh size goes to zero, the quantity 
\[
\lim_{n\to \infty}\sum_{[u,v]\in \mathcal{P}^n([s,t])}A_{u,v}
\]
converges to the same limit, we note
\[
(\mathcal{I} A)_{s,t}:=\lim_{n\to \infty}\sum_{[u,v]\in \mathcal{P}^n([s,t])}A_{u,v}.
\]

\begin{lemma}[Sewing]
Let $0<\alpha\leq 1<\beta$. Then for any $A\in C^{\alpha, \beta}_2(E)$, $(\mathcal{I} A)$ is well defined. Moreover, denoting $(\mathcal{I} A)_t:=(\mathcal{I} A)_{0,t}$, we have $(\mathcal{I} A)\in C^\alpha([0,T], E)$ and $(\mathcal{I} A)_0=0$ and for some constant $c>0$ depending only on $\beta$ we have
\[
\norm{(\mathcal{I} A)_{t}-(\mathcal{I} A)_{s}-A_{s,t}}_{E}\leq c\norm{\delta A}_\beta |t-s|^\beta.
\]
\label{sewing}
\end{lemma}
We moreover cite the stochastic version of the Sewing Lemma due to \cite{stochasticsewing}:
\begin{lemma}[Stochastic Sewing]
Let $(\Omega, \mathcal{F}, \mathbb{P}, (\mathcal{F}_t)_{t\in [0,T]}$ be a complete filtered probability space. Let $p\geq 2$ and let $A:\Delta_2\to \R^d$ be such that for any $s\leq t$, $A_{s,t}$ is $\mathcal{F}_t$ measurable and $A_{s,t}\in L^p(\Omega)$. Suppose moreover that for some $\epsilon_1, \epsilon_2>0$, we have
\begin{equation}
    \begin{split}
        \norm{\mathbb{E}[(\delta A)_{s,u,t}|\mathcal{F}_s]}_{L^p(\Omega)}&\leq\Gamma_1 (t-s)^{1+\epsilon_1}\\
        \norm{(\delta A)_{s,u,t}}_{L^p(\Omega)}&\leq \Gamma_2 (t-s)^{1/2+\epsilon_2}.
        \label{stoch sewing conditions}
    \end{split}
\end{equation}
Then there exists a unique (up to a modification) $(\mathcal{F}_t)_{t\in [0,T]}$-adapted and $L^p(\Omega)$ integrable stochastic process $(\mathcal{A}_t)_{t\in [0,T]}$ with values in $\R^d$ such that $\mathcal{A}_0=0$ and such that 
\begin{equation}
\begin{split}
\label{stoch sewing a priori}
    \norm{\mathcal{A}_{s,t}-A_{s,t}}_{L^p(\Omega)}&\lesssim \Gamma_1 (t-s)^{1+\epsilon_1}+\Gamma_2(t-s)^{1/2+\epsilon_2}\\
    \norm{\mathbb{E}[\mathcal{A}_{s,t}-A_{s,t}|\mathcal{F}_s]}_{L^p(\Omega)}&\lesssim \Gamma_2|t-s|^{1+\epsilon_1}.
\end{split}
\end{equation}
Finally, for any sequence $(\mathcal{P}^n([s,t]))_n$ of  partitions of $[s,t]$, the sequence $(\mathcal{A}^n_{s,t})_n$ defined by 
\[
\mathcal{A}^n_{s,t}:=\sum_{[u,v]\in \mathcal{P}^n([s,t])} A_{u,v}
\]
converges in $L^p(\Omega)$ to $\mathcal{A}_{s,t}$.
\end{lemma}
Let us finally cite a result allowing to commute limits and sewings. 
\begin{lemma}[Lemma A.2 \cite{Galeati2021}]
\label{sewing convergence}
For $0<\alpha\leq 1<\beta$ and $E$ a Banach space, let  $A\in C^{\alpha, \beta}_2(E)$ and $ (A^n)_n\subset C^{\alpha, \beta}_2(E)$ such that for some $R>0$ $\sup_{n\in \mathbb{N}}\norm{\delta A^n}_\beta\leq R$ and such that $\norm{A^n-A}_\alpha\to 0$. Then \[\norm{\mathcal{I}(A-A^n)}_\alpha\to 0.\]
Moreover, adapting the proof in \cite{Galeati2021} in conjunction with the a priori bounds \eqref{stoch sewing a priori}, the result canonically extends to the Stochastic Sewing setting: Let $(\bar{\Omega}, \bar{\mathcal{F}}, \bar{\mathbb{P}} ,(\bar{\mathcal{F}}_t)_t)$ be a complete filtered probability space. Let $(A^n)_n$ be a sequence of  $\bar{\mathcal{F}}_t$ adapted processes $A^n:\Delta_2\to \R^d$ such that $A_{s,t}\in L^p(\bar{\Omega})$. Suppose that the inequalities \eqref{stoch sewing conditions} hold uniformly in $n\in \mathbb{N}$ and suppose that $\norm{A^n}_{C^\gamma_tL^p(\bar{\Omega})}\to 0$ for some $\gamma\in (0,1)$. Then, if $\mathcal{A}^n$ denotes the stochastic sewing of $A^n$, we have that $\norm{\mathcal{A}^n}_{C^\gamma_tL^p(\bar{\Omega})}\to 0$.
\end{lemma}
\subsection*{Identifications of Riemann integrals with sewings}
In the following, we establish the identity of Riemann integrals with corresponding sewings in our setting for a particular example. Other identifications follow similarly and their proofs are thus omitted.
\begin{lemma}[Identification]
\label{identification}
Let $\sigma_\epsilon: \R^d\to \mathbb{M}^{d\times n}$ be smooth and bounded. Let $X$ be a stochastic process satisfying 
\[
\sup_{s\neq t\in [0,T]}\frac{\mathbb{E}[|X_{s,t}|^m]}{|t-s|^{m\gamma/2}}<\infty
\]
for some $\gamma>1/2$. Moreover suppose $2H<(1+\frac{d}{(p/2)\wedge (4/3)})^{-1}$. Then $\mathbb{P}^H$ almost surely, the germ
\[
A_{s,t}^\epsilon=(\norm{\sigma_\epsilon}_{HS}^2*L_{s,t})(X_s)
\]
admits a stochastic sewing $\mathcal{A}^\epsilon$ and moreover, we have
\[
 \norm{\mathcal{A}^\epsilon_{0,t}-\left(\int_0^t \norm{\sigma_\epsilon( X_r-w^H_r)}_{HS}^2dr\right)}_{L^{m/2}(\Omega)}^{m/2}=0.
\]
\end{lemma}
\begin{proof}
As seen in the above proof to Lemma \ref{a priori bound}, the germ $A^\epsilon$ does admit a stochastic sewing. Moreover, it can be easily seen that for fixed $s\in [0,T]$, the germ $B^s_{u,v}=\norm{\sigma_\epsilon(u, X_s-w^H_u)}_{HS}^2(v-u)$ admits a stochastic sewing, for which we have by definition
\[
(\mathcal{I}B^s)_{s,t}=A^\epsilon_{s,t}
\]
understood as an equality in $L^{m/2}$. Similarly, the germ $C_{u,v}=\norm{\sigma_\epsilon(u, X_u-w^H_u)}_{HS}^2(v-u)$ admits a stochastic sewing, for which by definition, we have
\[
(\mathcal{I}C)_{s,t}=\int_s^t \norm{\sigma_\epsilon(u, X_r-w^H_r)}_{HS}^2dr
\]
and moreover $B^s_{s,t}=C_{s,t}$. We therefore conclude that 
\begin{align*}
    &\norm{ \int_s^t \norm{\sigma_\epsilon(r, X_r-w^H_r)}_{HS}^2dr-\mathcal{A}_{s,t}^\epsilon}_{L^{m/2}(\Omega)}\\
    &\lesssim \norm{(\mathcal{I}C)_{s,t}-C_{s,t}}_{L^{m/2}(\Omega)}+\norm{B^s_{s,t}-(\mathcal{I}B^s)_{s,t}}_{L^{m/2}(\Omega)}+\norm{A^\epsilon_{s,t}-\mathcal{A}^\epsilon_{s,t}}_{L^{m/2}(\Omega)}\\
    &=O(|t-s|^{1+\delta}).
\end{align*}
Hence, the function
\[
t\to \norm{ \int_0^t \norm{\sigma_\epsilon(r, X_r-w^H_r)}_{HS}^2dr-\mathcal{A}_{t}^\epsilon}_{L^{m/2}(\Omega)}
\]
is a constant allowing to conclude. 
\end{proof}

\subsection*{Extending the Ito integral in the presence of a perturbative fractional Brownian motion}
In the following, we demonstrate how to extend the definition of the stochastic integral in the presence of a perturbative fractional Brownian motion, i.e. we define the stochastic integral
\[
\int_0^t \sigma(X_r-w^H_r)dB_r
\]
for $\sigma\in L^p_x$. Note that for progressively measurable processes $X$ such that \[
\sup_{s\neq t\in [0,T]}\frac{\mathbb{E}[|X_{s,t}|^m]}{|t-s|^{m\gamma_0/2}}<\infty
\]
it is a priori not obvious why this stochastic integral should be well defined as the integrand is not an element of $L^2(\Omega\times [0,T])$. However, note that due to the previous Lemma \ref{identification}, we can pass by mollifications of $\sigma$ and subsequently exploit the regularizing effect due to the associated averaging operator on the level of the Ito isometry. 
\begin{lemma}
Let $\sigma\in L^p_x$ and $\sigma_\epsilon$ be any cut-off mollification. Let $X$ and $\mathcal{A}^\epsilon$ be as in Lemma \ref{identification}  above. There hold the robustified Ito isometry
\begin{equation}
    \mathbb{E}\left[ \left( \int_0^t \sigma_\epsilon( X_r-w^H_r)dB_r\right)^2\right]=\norm{\mathcal{A}^\epsilon_t}_{L^1(\Omega)}
    \label{ito iso}
\end{equation}
and for $m\geq 4$ the Burkholder-Davis-Gundy inequality
\begin{equation}
    \mathbb{E}\left[ \sup_{t\in [0,T]}\left( \int_0^t \sigma_\epsilon( X_r-w^H_r)dB_r\right)^m\right]\lesssim\norm{\mathcal{A}^\epsilon_T}_{L^{m/2}(\Omega)} 
    \label{BDGII}
\end{equation}
$\mathbb{P}^H$-almost surely. In particular, the sequence $\left( \int_0^{(\cdot)} \sigma_\epsilon( X_r-w^H_r)dB_r\right)_\epsilon$ is $\mathbb{P}^H$-almost surely Cauchy in $L^{m/2}(\Omega, C([0,T]))$, whose limit we denote by 
\[
t\to I_t\sigma=\int_0^t \sigma(X_r-w^H_r)dB_r.
\]
The construction is independent of the cut-off mollification chosen and adapted to the filtration generated by $B$. It is a martingale with respect to that filtration.  Moreover the so constructed integral is linear in the sense that for $\sigma_1, \sigma_2$, we have
\[
I_t(\sigma_1+\sigma_2)=I_t\sigma_1+I_t\sigma_2.
\]
\label{extending the integral}
\end{lemma}
\begin{proof}
The above \eqref{ito iso} and \eqref{BDGII} are immediate consequences of the classical Ito isometry and Burkholder-Davis-Gundy inequality available in this setting as well as the previous Lemma \ref{identification}. Moreover, for $\epsilon, \epsilon'>0$, we have similarly
\[
\mathbb{E}\left[ \sup_{t\in [0,T]}\left( \int_0^t (\sigma_\epsilon( X_r-w^H_r)-\sigma_{\epsilon'}( X_r-w^H_r))dB_r\right)^m\right]\lesssim\norm{\mathcal{A}^{\epsilon, \epsilon'}_T}_{L^{m/2}(\Omega)} 
\]
where $\mathcal{A}^{\epsilon, \epsilon'}$ denotes the stochastic Sewing of the germ
\[
A^{\epsilon, \epsilon'}_{s,t}=(\norm{\sigma_{\epsilon}-\sigma_{\epsilon'}}_{HS}^2*L_{s,t})(X_s)
\]
Again similarly to Lemma \ref{a priori bound} one shows that
\[
\norm{A^{\epsilon, \epsilon'}_{s,t}}_{L^{m/2}(\Omega)}\lesssim \norm{\sigma_{\epsilon}-\sigma_{\epsilon'}}_{L^p}|t-s|^{\gamma_0/2}
\]
as well as 
\[
\norm{(\delta A^{\epsilon, \epsilon'})_{s,u,t}}_{L^{m/2}(\Omega)}\lesssim \norm{\sigma_{\epsilon}-\sigma_{\epsilon'}}_{L^p}|t-s|^{\gamma_0/2+\gamma_1}.
\]
From the a priori bound \eqref{stoch sewing a priori} available in the stochastic sewing lemma and the fact that $\mathcal{A}^{\epsilon, \epsilon'}_0=0$, we infer that
\begin{align*}
    \norm{\mathcal{A}^{\epsilon, \epsilon'}_T}_{L^{m/2}(\Omega)}&\lesssim \norm{\mathcal{A}^{\epsilon, \epsilon'}}_{C^{\gamma_0/2}_t L^{m/2}(\Omega)}\\
    &\lesssim \norm{A^{\epsilon, \epsilon'}}_{C^{\gamma_0/2}_t L^{m/2}(\Omega)}+\norm{(\delta A^{\epsilon, \epsilon'})}_{C^{\gamma_0/2+\gamma_1}_tL^{m/2}(\Omega)}\lesssim  \norm{\sigma_{\epsilon}-\sigma_{\epsilon'}}_{L^p}.
\end{align*}
We conclude thus that indeed the sequence $\left( \int_0^{(\cdot)} \sigma_\epsilon( X_r-w^H_r)dB_r\right)_\epsilon$ is $\mathbb{P}^H$-almost surely Cauchy in $L^{m/2}(\Omega, C([0,T]))$, allowing to define the corresponding limit as the extended stochastic integral. Remark moreover that this construction is independent of the sequence of cut-off mollifications chosen, which is immediate by replacing $\sigma_{\epsilon'}$ by $\sigma$ in the above considerations. Adaptedness follows from the fact that the sequence of approximations is adapted by classical Ito theory. Linearity follows from the fact that cut-off mollifications are linear, i.e. $(\sigma_1+\sigma_2)^\epsilon=\sigma_1^\epsilon+\sigma_2^\epsilon$ as well as the fact that the classical Ito integral is linear.  Finally, the martingale property follows from Lemma \ref{step one conv} and Corollary \ref{step two conv} above.

\end{proof}

\section*{Acknowledgement}The authors  acknowledge support from the European Research Council (ERC) under the European Union’s Horizon 2020 research
and innovation programme (grant agreement No. 949981)

	\begin{figure}[ht]
	    \centering
	  \includegraphics[height=10mm]{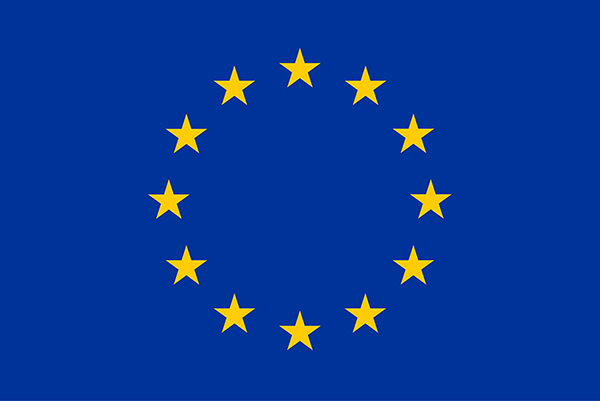}
	\end{figure}

\bibliographystyle{alpha}
\bibliography{main}

\begin{thebibliography}{BFH18}

\bibitem[BFH18]{BreitFeireislHofmanova}
Dominic Breit, Eduard Feireisl, and Martina Hofmanová.
\newblock {\em Stochastically Forced Compressible Fluid Flows}.
\newblock De Gruyter, 2018.

\bibitem[CG16]{gubicat}
R.~Catellier and M.~Gubinelli.
\newblock Averaging along irregular curves and regularisation of {ODE}s.
\newblock {\em Stochastic Processes and their Applications}, 126(8):2323--2366,
  2016.

\bibitem[FH14]{frizhairer}
Peter~K. Friz and Martin Hairer.
\newblock {\em A course on rough paths}.
\newblock Universitext. Springer, Cham, 2014.

\bibitem[Gal21]{Galeati2021}
Lucio Galeati.
\newblock Nonlinear young differential equations: A review.
\newblock {\em Journal of Dynamics and Differential Equations}, February 2021.

\bibitem[GG21]{galeati2020noiseless}
Lucio Galeati and Massimiliano Gubinelli.
\newblock Noiseless regularisation by noise.
\newblock {\em Revista Matem{\'{a}}tica Iberoamericana}, 38(2):433--502, July
  2021.

\bibitem[GH80]{horowitz}
Donald Geman and Joseph Horowitz.
\newblock {Occupation Densities}.
\newblock {\em The Annals of Probability}, 8(1):1 -- 67, 1980.

\bibitem[GH20]{galeati2020regularizationn}
Lucio Galeati and Fabian~A. Harang.
\newblock Regularization of multiplicative sdes through additive noise, 2020.
\newblock arXiv:2008.02335.

\bibitem[Gub04]{gubi}
M~Gubinelli.
\newblock Controlling rough paths.
\newblock {\em J. Func. Anal.}, 216(1):86 -- 140, 2004.

\bibitem[HL20]{hairer2019averaging}
Martin Hairer and Xue-Mei Li.
\newblock {Averaging dynamics driven by fractional Brownian motion}.
\newblock {\em The Annals of Probability}, 48(4):1826 -- 1860, 2020.

\bibitem[Hof13]{HOFMANOVA20134294}
Martina Hofmanová.
\newblock Degenerate parabolic stochastic partial differential equations.
\newblock {\em Stochastic Processes and their Applications},
  123(12):4294--4336, 2013.

\bibitem[HP21]{harang2020cinfinity}
F.~Harang and N.~Perkowski.
\newblock C$\infty$- regularization of {ODEs} perturbed by noise.
\newblock {\em Stochastics and Dynamics}, page 2140010, 2021.

\bibitem[HS12]{Hofmanov2012}
Martina Hofmanov{\'{a}} and Jan Seidler.
\newblock On weak solutions of stochastic differential equations.
\newblock {\em Stochastic Analysis and Applications}, 30(1):100--121, January
  2012.

\bibitem[HS13]{Hofmanov20133}
Martina Hofmanov{\'{a}} and Jan Seidler.
\newblock On weak solutions of stochastic differential equations {II}.
\newblock {\em Stochastic Analysis and Applications}, 31(4):663--670, July
  2013.

\bibitem[KS21]{kuhn2021convolution}
Franziska Kühn and René~L. Schilling.
\newblock Convolution inequalities for besov and triebel--lizorkin spaces, and
  applications to convolution semigroups, 2021.
\newblock arXiv:2101.03886.

\bibitem[Lê20]{stochasticsewing}
Khoa Lê.
\newblock {A stochastic sewing lemma and applications}.
\newblock {\em Electronic Journal of Probability}, 25(none):1 -- 55, 2020.

\bibitem[{\O}ks03]{oksendal2003}
Bernt {\O}ksendal.
\newblock {\em Stochastic Differential Equations}.
\newblock Springer Berlin Heidelberg, 2003.

\bibitem[RY99]{Revuz1999}
D.~Revuz and M.~Yor.
\newblock {\em Continuous Martingales and Brownian Motion}.
\newblock Springer Berlin Heidelberg, 1999.

\bibitem[Sko61]{skoro}
A.~V. Skorokhod.
\newblock On existence and uniqueness of solutions to stochastic diffusion
  equations.
\newblock {\em Sibirsk. Mat. Zh.}, 2:129–137, 1961.

\end{thebibliography}

\end{document}